\documentclass[12pt]{amsart}

\usepackage{amsrefs}
\usepackage{enumerate}
\usepackage[hidelinks,colorlinks]{hyperref}

\hypersetup{
    linkcolor=red,
    citecolor=blue,
 }

\usepackage{amsmath}
\usepackage{amssymb}
\usepackage{mathtools}
\usepackage{mathrsfs}
\usepackage{color}

\usepackage[normalem]{ulem}
\usepackage{comment}
\usepackage{nicefrac}
\usepackage{amsthm}

\usepackage{lipsum}

\usepackage{xcolor}

\newtheorem{theorem}{Theorem}
\newtheorem{prop}{Proposition}
\newtheorem{lemma}{Lemma}
\newtheorem{cor}{Corollary}

\newtheorem{definition}{Definition}

\newcommand{\RR}{\mathbb{R}}
\newcommand{\NN}{\mathbb{N}}

\newcommand{\ZZ}{\mathbb{Z}}
\newcommand{\CC}{\mathbb{C}}

\newsavebox\foobox
\newlength{\foodim}
\newcommand{\slantbox}[2][0]{\mbox{%
        \sbox{\foobox}{#2}%
        \foodim=#1\wd\foobox
        \hskip \wd\foobox
        \hskip -0.5\foodim
        \pdfsave
        \pdfsetmatrix{1 0 #1 1}%
        \llap{\usebox{\foobox}}%
        \pdfrestore
        \hskip 0.5\foodim
}}

\newcommand{\laplace}{\slantbox[-.45]{$\mathscr{L}$}}
\newcommand{\invlaplace}{\laplace^{-1}}
\newcommand{\fourier}{\mathcal{F}}
\newcommand{\invfourier}{\mathcal{F}^{-1}}

\newcommand{\defeq}{\vcentcolon=}
\newcommand{\eqdef}{=\vcentcolon}

\usepackage{graphicx}

\begin{document}

\title[Convergence of Newton series and asymptotics of finite differences]{On the convergence of Newton series and the asymptotics of finite differences}

\begin{abstract}
     Suppose a complex function $f$ has a Lebesgue measurable inverse Laplace transform. We show that the $n$th order forward and backward differences of $f$ at $z_0\in\CC$ tend to zero as $n\to\infty$ whenever $z_0$ lies in the region of absolute convergence of $f$. Under the same hypothesis, we show that the Newton series of $f$ centered at $z_0$ exists and converges in the half-plane $\Re(z)>\Re(z_0)$. Assuming instead that $f$ has a Lebesgue measurable inverse Fourier transform, we show that the $n$th order forward, backward, and central differences of $f$ at any $y\in\RR$ are $o(2^n)$. Consequently, we show that the binomial sum $\sum_{k\geq0}{n\choose k}f(k)$ is $o(2^n)$.
\end{abstract}

\author{Glenn Bruda}
\thanks{\textit{Email:} \texttt{glenn.bruda@ufl.edu}}
\date{\today}
\maketitle

\section{Introduction}

For a complex function $f$, we define the forward, backward, and central differences of $f$ at $z\in\CC$, denoted $\Delta_{h}^{n}[f](z)$, $\nabla_{h}^{n}[f](z)$, $ \delta_{h}^{n}[f](z)$ respectively, as
\begin{align*}
    \Delta_{h}^{n}[f](z)&\defeq\sum_{k=0}^{n}{n\choose k}f(z+kh)(-1)^{n-k},\\
    \nabla_{h}^{n}[f](z)&\defeq\sum_{k=0}^{n}{n\choose k}f(z-kh)(-1)^{k},\\
    \delta_{h}^{n}[f](z)&\defeq\sum_{k=0}^{n}{n\choose k}f\left(z+\left(\frac{n}{2}-k\right)h\right)(-1)^{k},
\end{align*}
where $h\in\RR^+$ and $n\in\mathbb{N}\defeq\{0,1,2,\cdots\}$. By convention, if the subscript $h$ of a finite difference is omitted, $h$ is assumed to be $1$.

In general, the limiting behavior of finite differences can be difficult to determine. In \cite{flajoletsedgewick}, Flajolet and Sedgewick used the N\o rlund-Rice integral\footnote{See \cite[Lemma 1]{flajoletsedgewick}.} to determine the asymptotics of many finite differences relevant to combinatorics and computer science. Being of interest to such fields, most of these differences do not converge to zero. For some of these differences, this can be surprising, especially when the function differenced vanishes at infinity. Furthermore, the finite differences of many functions, even those that are unbounded such as non-constant polynomials, do converge to zero. In this article, we clarify the seemingly unpredictable limiting behavior of forward and backward differences by analyzing a Laplace transform integral formula. In particular, we prove the following result:
\begin{prop}
    Assuming it exists, suppose the inverse Laplace transform of $f$ is a linear combination of Lebesgue measurable functions and elements of the set $$\left\{\delta^{(m)}(t-a):m\in\NN,a\in\RR^{\geq0}\right\},$$ where $\delta^{(m)}$ is the $m$th distributional derivative of the Dirac delta function. Let $z\in\CC$ lie in the region of absolute convergence of $f$ and let $h\in\RR^+$. Then
    \begin{align*}
        \lim_{n\to\infty}\Delta_{h}^{n}[f](z)=0.
    \end{align*}
\end{prop}

Similar analysis allows us to determine the existence and convergence of certain Newton series\footnote{Also known as \emph{Newton's forward difference formula}.}:
\begin{definition}[Newton series]
    For a suitable function $f$, the Newton series of $f$ centered at $z_0\in\CC$ is given by
    \begin{align*}
        f(z)=\sum_{k\geq0}{z-z_0\choose k}\Delta^{k}[f](z_0).
    \end{align*}
\end{definition}

Newton series, often described as a ``finite analog of a Taylor series'' \cite{wolfram}, can be tricky to prove the convergence of. The $n$th partial sum of a Newton series of a well-defined function $f$ centered at $z_0$ agrees with $f$ for $z\in\{z_0,z_0+1,\cdots,z_0+n\}$; as such, these partial sums are sometimes called \emph{Newton interpolating polynomials}, as seen in \cite{entireandexptype}. We are interested in where a Newton series agrees with $f$ within the context of the rest of the complex plane.

Though Newton series exist for a variety of analytic functions, there are curious exceptions (e.g., the canonical $\sin\pi z$). Using the aforementioned Laplace transform integral formula, we guarantee the existence and convergence of Newton series for certain functions:
\begin{prop}
    Assuming it exists, suppose the inverse Laplace transform of $f$ is a linear combination of Lebesgue measurable functions and elements of the set $$\left\{\delta^{(m)}(t-a):m\in\NN,a\in[0,\log2)\right\}.$$ Let $z_0\in\CC$ lie in the region of absolute convergence of $f$. Then the Newton series of $f$ centered at $z_0$ exists and converges for all $z\in\CC:\Re(z)>\Re(z_0)$.    
\end{prop}

This result induces the corollary that, for an appropriately chosen $z_0\in\CC$, any rational function $R$ with real coefficients admits a Newton series centered at $z_0$ which is convergent in the half-plane $\Re(z)>\Re(z_0)$. In particular, we can take $z_0$ to be any complex number such that its real part is strictly larger than the real parts of every pole of $R$ in the right half-plane. If $R$ has no poles in the right half-plane, then we must take $\Re(z_0)>0$.

By making the fairly strong assumption that $f$ has an inverse Laplace transform (alongside other conditions), we are able to show that the forward and backward differences of $f$ converge to zero. If we instead assume that $f$ has a Lebesgue measurable inverse Fourier transform, we find the $n$th order forward, backward, and central differences of $f$ at any $y\in\RR$ to be $o(2^n)$. A corollary to this gives a similar result for general $n$-term binomial sums:
\begin{prop}
     Assuming it exists, suppose the inverse Fourier transform of $f$ is a linear combination of Lebesgue measurable functions and elements of the set $$\{\delta^{(m)}(x-a):m\in\NN,a\in\RR\backslash\ZZ\}.$$ Then
    \begin{align*}
        \sum_{k=0}^{n}{n\choose k}f(k)=o(2^n).
    \end{align*}    
\end{prop}

The results of this article are split into two parts: an analysis of finite differences and Newton series using a Laplace transform integral formula and an analysis of finite differences using Fourier transform integral formulas. Following this portion, there is a short discussion on the applications and relevance of these results.

\section{Analysis of a Laplace transform integral formula}

\begin{definition}[Laplace regions of conditional and absolute convergence]
    Suppose $f$ has an inverse Laplace transform. Then the Laplace region of {conditional} convergence of $f$, denoted $\mathcal{C}_{L}(f)$, is the set
    \begin{align*}
        \mathcal{C}_{L}(f)\defeq\left\{z\in\CC:\int_{0}^{\infty}e^{-zt}\invlaplace f(t)dt<\infty\right\}.
    \end{align*}
    The Laplace region of {absolute} convergence of $f$, denoted $\mathcal{A}_{L}(f)$, is the set
    \begin{align*}
        \mathcal{A}_{L}(f)\defeq\left\{z\in\CC:\int_{0}^{\infty}\left|e^{-zt}\invlaplace f(t)\right|dt<\infty\right\}.
    \end{align*}
\end{definition}

\begin{lemma}
    Suppose $f$ has an inverse Laplace transform. Let $n\in\NN$, $h\in\RR^+$, and $z\in\mathcal{C}_{L}(f)$. Then
    \begin{align*}
        \Delta_{h}^{n}[f](z)=\int_{0}^{\infty}e^{-zt}(e^{-ht}-1)^n\invlaplace f(t)dt.
    \end{align*}
\end{lemma}

\begin{proof}    
Let $n\in\NN$ and $h\in\RR^+$. Let $f\in\invlaplace$. Let $z\in\mathcal{C}_{L}(f)$. Denoting $g=\invlaplace f$ for brevity, note $\laplace g=f$. So we have that
\begin{align*}
    \Delta_{h}^{n}[\laplace g](z)&=\sum_{k=0}^{n}{n\choose k}(-1)^{n-k}\laplace g(z+kh)\\
    &=\sum_{k=0}^{n}{n\choose k}(-1)^{n-k}\int_{0}^{\infty}e^{-(z+kh)t}g(t)dt\\
    &=\int_{0}^{\infty}e^{-zt}g(t)\sum_{k=0}^{n}{n\choose k}(-1)^{n-k}e^{-htk}dt\\
    &=\int_{0}^{\infty}e^{-zt}(e^{-ht}-1)^ng(t)dt.
\end{align*}
Thus,
\begin{align*}
    \Delta_{h}^{n}[f](z)=\int_{0}^{\infty}e^{-zt}(e^{-ht}-1)^n\invlaplace f(t)dt,
\end{align*}
as desired.
\end{proof}

One of the two applications of Lemma 1 arises in proving the following theorem, wherein we show when $\lim_{n\to\infty}\Delta^{n}_{h}[f](z)=0$. To do this, we assume that $z\in\mathcal{A}_{L}(f)$; by \cite[Subsection 2.7.3]{analyticinROAC}, this implies that $f$ is analytic at $z$. Since finite differences are a way of approximating derivatives, it is not surprising that $n$th order forward differences converge when we make assumptions on $f$ that imply analyticity. However, analyticity is not enough, as shown in \cite[Example 2]{flajoletsedgewick}.

\begin{definition}
    For $m\in\NN$, we write $\delta^{(m)}$ to denote the $m$th distributional derivative of the Dirac delta function.
\end{definition}

\begin{theorem}
    Assuming it exists, suppose that $\invlaplace f$ is a linear combination of Lebesgue measurable functions and elements of the set $\left\{\delta^{(m)}(t-a):m\in\NN,a\in\RR^{\geq0}\right\}$. Let $z\in\mathcal{A}_{L}(f)$ and $h\in\RR^+$. Then
    \begin{align*}
        \lim_{n\to\infty}\Delta_{h}^{n}[f](z)=0.
    \end{align*}
\end{theorem}

\begin{proof}
    Let $z\in\mathcal{A}_{L}(f)$. Let $h\in\RR^+$. Suppose that $\invlaplace f$ is a linear combination of Lebesgue measurable functions and elements of the set $\left\{\delta^{(m)}(t-a):m\in\NN,a\in\RR^{\geq0}\right\}$; that is,
    \begin{align*}
        \invlaplace f(t)=\sum_{k=0}^{u}a_k g_k(t)+\sum_{k=0}^{j}b_k\delta^{(k)}(t-c_k),
    \end{align*}
    where $u,j\in\NN$, each $g_k$ is a Lebesgue measurable function, $(a_k),(b_k)$ are sequences from $\CC$, and $(c_k)$ is a sequence from $\RR^{\geq0}$. As a finite sum of Lebesgue measurable functions is Lebesgue measurable, without loss of generality we have that
    \begin{align*}
        \invlaplace f(t)=m(t)+\sum_{k=0}^{j}b_k\delta^{(k)}(t-c_k)
    \end{align*}
    for some Lebesgue measurable function $m$. We know that $\laplace m$ exists at $z$ since $m$ is a difference of functions with a Laplace transform existing at $z$. So
    \begin{align*}
        f(z)=\laplace m(z)+\sum_{k=0}^{j}b_k e^{-c_k z} z^k.
    \end{align*}
    Let $p(z)=\sum_{k=0}^{j}b_k e^{-c_k z} z^k$. Then
     \begin{align*}
         \Delta_{h}^{n}[f](z)=\Delta_{h}^{n}[\laplace m](z)+\Delta_{h}^{n}[p](z).
     \end{align*}
     Let $N\in\{0,1,\cdots,j\}$. We now argue that $\Delta^{n}_{h}[e^{-c_N z}z^N](z)$ converges to zero. We have that
     \begin{align*}
         \Delta^{n}_{h}[e^{-c_N z}z^N](z)&=\sum_{k\geq0}{n\choose k}(-1)^{n-k} e^{-c_{N} (z+kh)} (z+kh)^N\\
         &=e^{-c_{N} z}\sum_{b\geq0}{N\choose b}z^{N-b}\sum_{k\geq0}{n\choose k}(-1)^{n-k}e^{-c_N kh}(kh)^{b}.
     \end{align*}
     Note
     \begin{align*}
         \sum_{k\geq0}{n\choose k}(-1)^{n-k}e^{-t kh}(kh)^{b}=(-1)^b\frac{d^b}{dt^{b}}\left(e^{-th}-1\right)^n=O(n^b(e^{-th}-1)^n).
     \end{align*}
     Thus by taking $t=c_N$,
     \begin{align*}
         \Delta^{n}_{h}[e^{-c_N z}z^N](z)&=e^{-c_{N} z}\sum_{b\geq0}{N\choose b}z^{N-b}O(n^b(e^{-c_Nh}-1)^n)\\
         &=O\left((e^{-c_Nh}-1)^n\right)\sum_{b\geq0}{N\choose b}z^{N-b}O(n^b)\\
         &=O\left((e^{-c_Nh}-1)^n\right)O\left(n^N\right)=O\left(n^N(e^{-c_Nh}-1)^n\right).
     \end{align*}
     As $|e^{-c_Nh}-1|<1$, it follows that $\lim_{n\to\infty}\Delta^{n}_{h}[e^{-c_N z}z^N](z)=0$. Thus, $\lim_{n\to\infty}\Delta^{n}_{h}[p](z)=0$.

    We now argue that $\Delta^{n}_{h}[\laplace m](z)$ converges to zero. As $z\in\mathcal{C}_{L}(\laplace m)$, by Lemma 1 we have that
    \begin{align*}
        \Delta_{h}^{n}[\laplace m](z)=\int_{0}^{\infty}e^{-zt}(e^{-ht}-1)^n m(t)dt.
    \end{align*}
    Note $|e^{-ht}-1|<1$ for all all $t>-h^{-1}\log2$, so $(e^{-ht}-1)^n$ converges pointwise to $0$ for all $t\geq0$. Let $j_n(t)\defeq(e^{-ht}-1)^ne^{-zt} m(t)$. We have that
    \begin{align*}
        |j_n(t)|\leq|e^{-zt}m(t)|
    \end{align*}
    for all $t\in\RR^{+}$ and $n\in\NN$. Note that since $m$ is Lebesgue measurable, $e^{-zt}m(t)$ and each $j_n$ are Lebesgue measurable. As $z\in\mathcal{A}_{L}(f)$ and $z\in\CC=\mathcal{A}_{L}(p)$, we have that $z\in\mathcal{A}_{L}(\laplace m)$. So $e^{-zt}m(t)$ is absolutely integrable on $\RR^{+}$. As $j_n(t)$ converges pointwise to $0$ for all $t\geq0$, by the dominated convergence theorem we have that
    \begin{align*}
        \lim_{n\to\infty}\Delta_{h}^{n}[\laplace m](z)=\lim_{n\to\infty}\int_{0}^{\infty}j_n(t)dt=0.
    \end{align*}
    Thus, 
     \begin{align*}
         \lim_{n\to\infty}\Delta_{h}^{n}[f](z)=\lim_{n\to\infty}\Delta_{h}^{n}[\laplace m](z)+\lim_{n\to\infty}\Delta_{h}^{n}[p](z)=0,
     \end{align*}
     as desired.
\end{proof}

Note that the assumption that $z\in\mathcal{A}_{L}(f)$ is necessary; $z\in\mathcal{C}_{L}(f)$ is not enough to guarantee convergence to zero. Indeed, consider the $n$th order forward differences of $f(z)=(1+z^2)^{-1/2}$ at $0$. By \cite[Appendix B]{laplacebessel}, we have that $\invlaplace f(t)=J_{0}(t)$, where $J_0$ is the zeroth order Bessel function of the first kind, and $0\in\mathcal{C}_{L}(f)$ since $J_0$ is integrable on $\RR^+$ per \cite[\href{https://dlmf.nist.gov/10.22\#E41}{Equation 10.22.41}]{NIST:DLMF}. However, \cite[pg.119]{flajoletsedgewick} showed that for some $c_1,c_2\in\RR$,
\begin{align*}
    (-1)^n\Delta^{n}[f](0)=\sum_{k\geq0}{n\choose k}\frac{(-1)^k}{\sqrt{k^2+1}}=1+c_1\sqrt{\log n}\cos(\log n+c_2)+O\left((\log n)^{-1/2}\right),
\end{align*}
which diverges to infinity. 

\begin{cor}
    Assuming it exists, suppose that $\invlaplace g$ is a linear combination of Lebesgue measurable functions and elements of the set $\left\{\delta^{(m)}(t-a):a\in\RR^{\geq0},m\in\NN\right\}$. Let $z\in\CC$ such that $-z\in \mathcal{A}_{L}(g)$. Let $h\in\RR^+$. Let $f(t)=g(-t)$. Then
    \begin{align*}
        \lim_{n\to\infty}\nabla_{h}^{n}[f](z)=0.
    \end{align*}
\end{cor}

\begin{proof}
    This result follows immediately from Theorem 1 by noting that
    \begin{align*}
        \nabla^{n}_{h}[f](z)=(-1)^n\Delta^{n}_{h}[g](-z).
    \end{align*}
\end{proof}

The second application of Lemma 1 allows us to guarantee the existence and convergence of certain Newton series, which can be a delicate task otherwise.

\begin{theorem}
    Assuming it exists, suppose that $\invlaplace f$ is a linear combination of Lebesgue measurable functions and elements of the set $\left\{\delta^{(m)}(t-a):m\in\NN,a\in[0,\log2)\right\}$. Let $z_0\in\mathcal{A}_{L}(f)$. Then the Newton series of $f$ centered at $z_0$ exists and converges in the half-plane $\Re(z)>\Re(z_0)$.
\end{theorem}

\begin{proof}
    Let $z_0\in\mathcal{A}_{L}(f)$. Let $z\in\CC:\Re(z)>\Re(z_0)$. Suppose that $\invlaplace f$ is a linear combination of Lebesgue measurable functions and elements of the set $\left\{\delta^{(m)}(t-a):m\in\NN,a\in[0,\log2)\right\}$. So, without loss of generality, we have that
    \begin{align*}
        \invlaplace f(t)=m(t)+\sum_{k=0}^{j}a_k\delta^{(k)}(t-b_k),
    \end{align*}
    where $j\in\NN$, $(a_k)\subseteq\CC$, $(b_k)\subseteq[0,\log2)$, and $m$ is some Lebesgue measurable function. We know that $\laplace m$ exists at $z_0$ since $m$ is a difference of functions with a Laplace transform existing at $z_0$. So
    \begin{align*}
        f(z_0)=\laplace m(z_0)+\sum_{k=0}^{j}a_k e^{-b_k z_0}z_0^k.
    \end{align*}
    Let $p(z)=\sum_{k=0}^{j}a_k e^{-b_k z}z^k$. Then
     \begin{align*}
         \Delta^{k}[f](z_0)=\Delta^{k}[\laplace m](z_0)+\Delta^{k}[p](z_0).
     \end{align*}
    So
    \begin{align*}
        \sum_{k\geq0}{z-z_0\choose k}\Delta^{k}[f](z_0)=\sum_{k\geq0}{z-z_0\choose k}\Delta^{k}[\laplace m](z_0)+\sum_{k\geq0}{z-z_0\choose k}\Delta^{k}[p](z_0).\\
        %=\laplace m(z)+\sum_{k\geq0}{z-z_0\choose k}\Delta^{k}[p](z_0).
    \end{align*}
    Let $g(z)=p(z+z_0)$. Note that $g$ is entire and is of exponential type less than $\log2$ since each $|b_k|<\log2$. Thus by \cite[pg.127]{entireandexptype}, $g(z)=\sum_{k\geq0}{z\choose k}\Delta^{k}[g](0)=\sum_{k\geq0}{z\choose k}\Delta^{k}[p](z_0)$ for all $z\in\CC$. So $p(z)=g(z-z_0)=\sum_{k\geq0}{z-z_0\choose k}\Delta^{k}[p](z_0)$.
     Thus,
     \begin{align*}
         \sum_{k\geq0}{z-z_0\choose k}\Delta^{k}[f](z_0)=p(z)+\sum_{k\geq0}{z-z_0\choose k}\Delta^{k}[\laplace m](z_0).
     \end{align*}
     
     We now argue that $\sum_{k\geq0}{z-z_0\choose k}\Delta^{k}[\laplace m](z_0)=\laplace m(z)$. Since $z_0\in\mathcal{C}_{L}(\laplace m)$, we have by Lemma 1 that
    \begin{align*}
        \sum_{k\geq0}{z-z_0\choose k}\Delta^{k}[\laplace m](z_0)=\sum_{k\geq0}{z-z_0\choose k}\int_{0}^{\infty}(e^{-t}-1)^ke^{-z_0t}m(t)dt.
    \end{align*}
    We now justify the interchange of the sum and integral. As $m$ is Lebesgue measurable, it is sufficient to show that
    \begin{align*}
        \int_{0}^{\infty}\left|e^{-z_0t}m(t)\right|\sum_{k\geq0}\left|{z-z_0\choose k}(e^{-t}-1)^k\right|dt<\infty.
    \end{align*}
    For all $t\geq0$, $|e^{-t}-1|<1$; thus,
    \begin{align*}
        \sum_{k\geq0}\left|{z-z_0\choose k}(e^{-t}-1)^k\right|\leq\sum_{k\geq0}\left|{z-z_0\choose k}\right|\eqdef T(z,z_0).
    \end{align*}
    Thus,
    \begin{align*}
         \int_{0}^{\infty}\left|e^{-z_0t}m(t)\right|\sum_{k\geq0}\left|{z-z_0\choose k}(e^{-t}-1)^k\right|dt\leq \int_{0}^{\infty}T(z,z_0)\left|e^{-z_0t}m(t)\right|dt.
    \end{align*}
    The series $T(z,z_0)$ is known to converge whenever $\Re(z)>\Re(z_0)$; see \cite{binomialseries}. As $z_0\in\mathcal{A}_{L}(f)$ and $z_0\in\CC=\mathcal{A}_{L}(p)$, we have that $z_0\in\mathcal{A}_{L}(\laplace m)$. So, as $T(z,z_0)$ is constant with respect to $t$,
    \begin{align*}
        \int_{0}^{\infty}T(z,z_0)\left|e^{-z_0t}m(t)\right|dt<\infty.
    \end{align*}
    So
    \begin{align*}
        \int_{0}^{\infty}\left|e^{-z_0t}m(t)\right|\sum_{k\geq0}\left|{z-z_0\choose k}(e^{-t}-1)^k\right|dt<\infty.
    \end{align*}
    Thus, we can interchange the sum and integral; consequently,
    \begin{align*}
        &\sum_{k\geq0}{z-z_0\choose k}\int_{0}^{\infty}(e^{-t}-1)^ke^{-z_0t}m(t)dt\\
        &=\int_{0}^{\infty}e^{-z_0t}m(t)\sum_{k\geq0}{z-z_0\choose k}(e^{-t}-1)^kdt\\
        &=\int_{0}^{\infty}e^{-z_0t}m(t)\left((e^{-t}-1)+1\right)^{z-z_0}dt\\
        &=\int_{0}^{\infty}e^{-zt}m(t)dt.
    \end{align*}
    Since $z_0\in\mathcal{A}_{L}(\laplace m)\subseteq\mathcal{C}_{L}(\laplace m)$ and $\Re(z)>\Re(z_0)$,
    \begin{align*}
        \int_{0}^{\infty}e^{-zt}m(t)dt=\laplace m(z).
    \end{align*}
    So
    \begin{align*}
        \sum_{k\geq0}{z-z_0\choose k}\Delta^{k}[\laplace m](z_0)=\laplace m(z).
    \end{align*}
    Thus,
     \begin{align*}
         \sum_{k\geq0}{z-z_0\choose k}\Delta^{k}[f](z_0)=p(z)+\laplace m(z)=f(z),
     \end{align*}
     as desired.
\end{proof}

\begin{cor}
    Let $R$ be a rational function with real coefficients. Let $S$ be the set of poles of $R$. Let $z_0\in\CC$ such that $\Re(z_0)>\max\{0,\Re(\sigma):\sigma\in S\}$. Then the Newton series of $R$ centered at $z_0$ exists and converges in the half-plane $\Re(z)>\Re(z_0)$. 
\end{cor}

\begin{proof}
    Let $R=P/Q$ be a rational function with real coefficients. Let $S$ be the set of poles of $R$. Let $z_0\in\CC$ such that $\Re(z_0)>\max\{0,\Re(\sigma):\sigma\in S\}$. Let us decompose $R$ as
    \begin{align*}
        R(z)=p(z)+\sum_{k}\frac{P_k(z)}{Q_k(z)},
    \end{align*}
    where $p$ is a polynomial, and each $P_k/Q_k$ is a proper\footnote{A rational function $p/q$ is said to be \emph{proper} if $\deg(p)<\deg(q)$.} rational function. By this decomposition, it is sufficient to show that the Newton series of $p$ and each $P_k/Q_k$ exist and converge for all $z\in\CC:\Re(z)>\Re(z_0)$.

    A Newton series of any polynomial centered anywhere exists and converges in all of $\CC$. This follows from \cite[pg.127]{entireandexptype} since all polynomials and entire and are of exponential type less than $\log2$. Thus, the Newton series of $p$ centered at $z_0$ exists and converges for all $z\in\CC:\Re(z)>\Re(z_0)$.

    As each $P_k/Q_k$ is a proper rational function with real coefficients, by \cite[Theorem 4.1]{rationalinvlaplace} each $\invlaplace (P_k/Q_k)(t)$ is a finite linear combination of functions from the set
    \begin{align*}
        \left\{t^{k_1}e^{k_2t}\sin(k_3t),t^{j_1}e^{j_2t}\cos(j_3t)\bigg|k_1,j_1\in\NN,k_2,j_2\leq\max\{0,\Re(\sigma):\sigma\in S\},k_3,j_3\in\RR\right\}.
    \end{align*}
    Note that all of these functions are Lebesgue measurable. Taking $k_1,j_1\in\NN,k_2,j_2\leq\max\{0,\Re(\sigma):\sigma\in S\}$, and $k_3,j_3\in\RR$ as seen in the above set, we see that since $k_2-\Re(z_0)<0$,
    \begin{align*}
        \int_{0}^{\infty}|t^{k_1}e^{k_2t}\sin(k_3t)e^{-z_0t}|dt\leq\int_{0}^{\infty}t^{k_1}e^{\left(k_2-\Re (z_0)\right)t}dt<\infty.
    \end{align*}
    The case with cosine is identical; thus, by linearity, $z_0\in\mathcal{A}_{L}(P_k/Q_k)$ for each $k$. Thus by Theorem 2, the Newton series of each $P_k/Q_k$ centered at $z_0$ exists and converges for all $z\in\CC:\Re(z)>\Re(z_0)$. Thus, the Newton series of $R$ centered at $z_0$ exists and converges for all $z\in\CC:\Re(z)>\Re(z_0)$.
\end{proof}

The region of convergence claimed in Corollary 2 may be extended if $R$ has a linear denominator; see \cite[Theorem 1]{entireandexptype}.

\section{Analysis of Fourier transform integral formulas}

We now use Fourier transform integral formulas to yield results about the asymptotic behavior of finite differences (and eventually general binomial sums).

\begin{definition}
    In this article, we use the following convention for the definition of the Fourier transform of a function $f$ with suitable conditions:
    \begin{align*}
        \fourier f(z)\defeq\int_{\RR}f(x)e^{-2\pi izx}dx.
    \end{align*}
\end{definition}

\begin{lemma}
    Suppose $f$ has an inverse Fourier transform. Let $n\in\NN$, $h\in\RR^+$, and $z\in\CC$. Then
    \begin{flalign*}
        \delta^{n}_{h}[f](z)&=\int_{\RR}(1-e^{2\pi ihx})^ne^{-i\pi hxn}e^{-2\pi izx}\invfourier f(x)dx,\\
        \Delta^{n}_{h}[f](z)&=\int_{\RR}(e^{-2\pi ihx}-1)^ne^{-2\pi izx}\invfourier f(x)dx.
        %\nabla^{n}_{h}[f](z)=\int_{\RR}(1-e^{2\pi ihx})^ne^{-2\pi izx}\invfourier f(x)dx.
    \end{flalign*}
\end{lemma}

\begin{proof}
    Let $n\in\NN$ and $h\in\RR^+$. Let $f\in\fourier$. Let $g=\invfourier f$. Let $z\in\CC$. Then we have that
\begin{align*}
    \delta_{h}^{n}[\fourier g](z)&=\sum_{k=0}^{n}{n\choose k}\fourier g\left(z+\left(\frac{n}{2}-k\right)h\right)(-1)^{k}\\
    &=\sum_{k=0}^{n}{n\choose k}(-1)^k\int_{\RR}g(x)\exp\left({-2\pi i\left(z+\left(\frac{n}{2}-k\right)h\right)x}\right)dx\\
    &=\int_{\RR}g(x)e^{-i\pi hxn}e^{-2\pi izx}\sum_{k=0}^{n}{n\choose k}(-1)^ke^{2\pi ihxk}dx\\
    &=\int_{\RR}g(x)e^{-i\pi hxn}e^{-2\pi izx}(1-e^{2\pi ihx})^ndx.
\end{align*}
Thus,
\begin{align*}
    \delta^{n}_{h}[f](z)=\int_{\RR}(1-e^{2\pi ihx})^ne^{-i\pi hxn}e^{-2\pi izx}\invfourier f(x)dx.
\end{align*}
We also have that
\begin{align*}
    \Delta_{h}^{n}[\fourier g](z)&=\sum_{k=0}^{n}{n\choose k}\fourier g\left(z+kh\right)(-1)^{n-k}\\
    &=\sum_{k=0}^{n}{n\choose k}(-1)^{n-k}\int_{\RR}g(x)\exp\left({-2\pi i\left(z+kh\right)x}\right)dx\\
    &=\int_{\RR}g(x)e^{-2\pi izx}\sum_{k=0}^{n}{n\choose k}(-1)^{n-k}e^{-2\pi ihxk}dx\\
    &=\int_{\RR}g(x)e^{-2\pi izx}(e^{-2\pi ihx}-1)^ndx.
\end{align*}
Thus,
\begin{align*}
    \Delta^{n}_{h}[f](z)=\int_{\RR}(e^{-2\pi ihx}-1)^ne^{-2\pi izx}\invfourier f(x)dx,
\end{align*}
as desired.
\end{proof}

\begin{definition}
    Let $a,b\in\CC$. We define
    \begin{align*}
        a+b\ZZ\defeq\{a+bn:n\in\ZZ\}.
    \end{align*}
\end{definition}

\noindent An application of Lemma 2 gives the following theorem:

\begin{theorem}
    Let $h\in\RR^+$. Assuming it exists, suppose that $\invfourier f$ is a linear combination of Lebesgue measurable functions and elements of the set $$\{\delta^{(m)}(x-a):m\in\NN,a\in\RR\backslash(h^{-1}/2+h^{-1}\ZZ)\}.$$ Let $y\in\RR$. Then
    \begin{align*}
        \Delta_{h}^{n}[f](y)=o(2^n),\\
        \nabla_{h}^{n}[f](y)=o(2^n),\\
        \delta_{h}^{n}[f](y)=o(2^n).
    \end{align*}
\end{theorem}

\begin{proof}
    Let $y\in\RR$. Let $h\in\RR^+$. Suppose that $\invfourier f$ is a linear combination of Lebesgue measurable functions and elements of the set $$\{\delta^{(m)}(x-a):m\in\NN,a\in\RR\backslash(h^{-1}/2+h^{-1}\ZZ)\}.$$ So, without loss of generality, we have that
    \begin{align*}
        \invfourier f(x)=m(x)+\sum_{k=0}^{j}a_k\delta^{(k)}(x-c_k),
    \end{align*}
    where $j\in\NN$, $(a_k)\subseteq\CC$, $(c_k)\subseteq\RR\backslash(h^{-1}/2+h^{-1}\ZZ)$, and $m$ is some Lebesgue measurable function. So 
    \begin{align*}
        f(y)=\fourier m(y)+\sum_{k=0}^{j}a_k e^{-2\pi ic_k y}(2\pi iy)^k.
    \end{align*}
    Let $p(y)=\sum_{k=0}^{j}a_k e^{-2\pi ic_k y}(2\pi iy)^k$. First, we address the claim for central differences. Let $N\in\{0,1,\cdots,j\}$. We now argue that $\delta^{n}_{h}[e^{-2\pi ic_N y}(2\pi iy)^N](y)$ is $o(2^n)$. We have that
     \begin{align*}
         &\delta^{n}_{h}[e^{-2\pi ic_N y}(2\pi iy)^N](y)=\sum_{k\geq0}{n\choose k}(-1)^{k} e^{-2\pi ic_N (y+(n/2-k)h)}(2\pi i(y+(n/2-k)h))^N\\
         &=e^{-2\pi ic_{N} (y+nh/2)}\sum_{b\geq0}{N\choose b}(2\pi i(y+nh/2))^{N-b}\sum_{k\geq0}{n\choose k}(-1)^{k}e^{2\pi ic_N kh}(-2\pi ikh)^{b}.
     \end{align*}
     Note
     \begin{align*}
         \sum_{k\geq0}{n\choose k}(-1)^{k}e^{2\pi ix kh}(-2\pi ikh)^{b}=(-1)^b\frac{d^b}{dx^{b}}\left(1-e^{2\pi ixh}\right)^n=O(n^b(e^{2\pi ixh}-1)^n).
     \end{align*}
     Thus by taking $x=c_N$,
     \begin{align*}
         \delta^{n}_{h}[e^{-2\pi ic_N y}(2\pi iy)^N](y)=e^{-2\pi ic_{N} (y+nh/2)}\sum_{b\geq0}{N\choose b}(2\pi i(y+nh/2))^{N-b}O(n^b(e^{2\pi ic_Nh}-1)^n)\\
         =e^{-2\pi ic_{N} (y+nh/2)}O((e^{2\pi ic_Nh}-1)^n)\sum_{b\geq0}{N\choose b}(2\pi i(y+nh/2))^{N-b}O(n^b)\\
         =e^{-\pi ic_{N}nh}O((e^{2\pi ic_Nh}-1)^n)O(n^N)=e^{-\pi ic_{N}nh}O(n^N(e^{2\pi ic_Nh}-1)^n).
     \end{align*}
    As $c_N\not\in h^{-1}/2+h^{-1}\ZZ$, it follows that $|e^{2\pi ic_N h}-1|<2$; thus, $\Delta^{n}_{h}[e^{-2\pi ic_N y}(2\pi iy)^N](y)=o(2^n)$. So $\delta^{n}_{h}[p](y)=o(2^n)$. 

    We now argue that $\delta^{n}_{h}[\fourier m](y)=o(2^n).$ By Lemma 2, we have that
    \begin{align*}
        \delta^{n}_{h}[\fourier m](y)&=\int_{\RR}(1-e^{2\pi ihx})^ne^{-i\pi hxn}e^{-2\pi iyx}m(x)dx\\
        &=\int_{\RR}(e^{-i\pi hx}-e^{i\pi hx})^ne^{-2\pi iyx}m(x)dx\\
        &=\int_{\RR}\left(-2i\sin(\pi hx)\right)^ne^{-2\pi iyx}m(x)dx\\
        &=(-2i)^n\int_{\RR}\sin^{n}(\pi hx)e^{-2\pi iyx}m(x)dx.
    \end{align*}
    Note
    \begin{align*}
        |\sin(\pi hx)|\in
    \begin{cases} 
        \{1\} & x=\frac{2k+1}{2h}~\text{for some~}k\in\ZZ \\
        [0,1) & \text{otherwise.}\\
   \end{cases}
    \end{align*}
    So $\left|\sin^{n}(\pi hx)\right|$ converges pointwise to
    \begin{align*}
        s(x)\defeq
        \begin{cases} 
        1 & x=\frac{2k+1}{2h}~\text{for some~}k\in\ZZ \\
        0 & \text{otherwise.}\\
   \end{cases}
    \end{align*}
    Let $j_n(x)\defeq\left|\sin^{n}(\pi hx)e^{-2\pi iyx}m(x)\right|$. We have that
    \begin{align*}
        |j_n(x)|\leq\left|e^{-2\pi iyx}m(x)\right|=\left|m(x)\right|
    \end{align*}
    for all $x\in\RR$ and $n\in\NN$. Note that since $m$ has Fourier transform $\fourier m$, $m$ is absolutely integrable. By assumption, $m$ is Lebesgue measurable, so each $j_n$ is Lebesgue measurable. As $j_n(x)$ converges pointwise to $s(x)\left|e^{-2\pi iyx}m(x)\right|=s(x)\left|m(x)\right|$, by the dominated convergence theorem we have that
    \begin{align*}
        \lim_{n\to\infty}\int_{\RR}j_n(x)dx=\int_{\RR}s(x)\left|m(x)\right|dx.
    \end{align*}
    As $s=0$ almost everywhere,
    \begin{align*}
        \int_{\RR}s(x)\left|m(x)\right|dx=0.
    \end{align*}
    Thus,
    \begin{align*}
        \lim_{n\to\infty}\left|\frac{\delta^{n}_{h}[\fourier m](y)}{2^n}\right|=\lim_{n\to\infty}\left|\int_{\RR}\sin^{n}(\pi hx)e^{-2\pi iyx}m(x)dx\right|\leq\lim_{n\to\infty}\int_{\RR}j_n(x)dx=0.
    \end{align*}
    Therefore
    \begin{align*}
        \lim_{n\to\infty}\frac{\delta^{n}_{h}[\fourier m](y)}{2^n}=0.
    \end{align*}
    Thus,
    \begin{align*}
        \delta_{h}^{n}[\fourier m](y)=o(2^n).
    \end{align*}
    Thus,
    \begin{align*}
        \delta_{h}^{n}[f](y)=\delta_{h}^{n}[\fourier m](y)+\delta_{h}^{n}[p](y)=o(2^n)+o(2^n)=o(2^n).
    \end{align*}

    Now, we address the claim for forward differences. We now argue that \\$\Delta^{n}_{h}[e^{-2\pi ic_N y}(2\pi iy)^N](y)$ is $o(2^n)$. We have that
     \begin{align*}
         &\Delta^{n}_{h}[e^{-2\pi ic_N y}(2\pi iy)^N](y)=\sum_{k\geq0}{n\choose k}(-1)^{n-k} e^{-2\pi ic_N (y+kh)}(2\pi i(y+kh))^N\\
         &=e^{-2\pi ic_{N} y}\sum_{b\geq0}{N\choose b}(2\pi iy)^{N-b}\sum_{k\geq0}{n\choose k}(-1)^{n-k}e^{-2\pi ic_N kh}(2\pi ikh)^{b}.
     \end{align*}
     Note
     \begin{align*}
         \sum_{k\geq0}{n\choose k}(-1)^{n-k}e^{-2\pi ix kh}(2\pi ikh)^{b}=(-1)^b\frac{d^b}{dx^{b}}\left(e^{-2\pi ixh}-1\right)^n=O(n^b(e^{-2\pi ixh}-1)^n).
     \end{align*}
     Thus by taking $x=c_N$,
     \begin{align*}
         \Delta^{n}_{h}[e^{-2\pi ic_N y}(2\pi iy)^N](y)&=e^{-2\pi ic_{N} y}\sum_{b\geq0}{N\choose b}(2\pi iy)^{N-b}O(n^b(e^{-2\pi ic_Nh}-1)^n)\\
         &=O((e^{-2\pi ic_Nh}-1)^n)\sum_{b\geq0}{N\choose b}(2\pi iy)^{N-b}O(n^b)\\
         &=O((e^{-2\pi ic_Nh}-1)^n)O(n^N)=O(n^N(e^{-2\pi ic_Nh}-1)^n).
     \end{align*}
     As $c_N\not\in h^{-1}/2+h^{-1}\ZZ$, it follows that $|e^{-2\pi ic_N h}-1|<2$; thus, $\Delta^{n}_{h}[e^{-2\pi ic_N y}(2\pi iy)^N](y)=o(2^n)$. As $p$ is a finite sum of terms which are $o(2^n)$, we have that $\Delta^{n}_{h}[p](y)=o(2^n)$.
    
    We now argue that $\Delta^{n}_{h}[\fourier m](y)=o(2^n)$. By Lemma 2, we have that
    \begin{align*}
        \lim_{n\to\infty}\left|\frac{\Delta^{n}_{h}[\fourier m](y)}{2^n}\right|&=\lim_{n\to\infty}\left|\int_{\RR}\left(\frac{e^{-2\pi ihx}-1}{2}\right)^ne^{-2\pi iyx}m(x)dx\right|\\
        &\leq\lim_{n\to\infty}\int_{\RR}\left|\left(\frac{e^{-2\pi ihx}-1}{2}\right)^ne^{-2\pi iyx}m(x)\right|dx.
    \end{align*}
    Let $$\ell_n(x)\defeq\left(\frac{e^{-2\pi ihx}-1}{2}\right)^n.$$ Note for each $n\in\NN$, $\ell_n(\RR)$ is contained in the closed unit disk, intersecting the unit circle only at $z=-1$. Specifically, for all $n\in\NN$, 
    \begin{align*}
        \ell^{-1}_n\left(\{-1\}\right)=\left\{x\in\RR:x=\frac{2k+1}{2h}~\text{for some~}k\in\ZZ\right\}.
    \end{align*}
    So $|\ell_n(x)|$ converges pointwise to 
    \begin{align*}
        s(x)\defeq
        \begin{cases} 
        1 & x=\frac{2k+1}{2h}~\text{for some~}k\in\ZZ \\
        0 & \text{otherwise.}\\
   \end{cases}
    \end{align*}
    The remaining argument follows identically to the argument given that $\delta^{n}_{h}[\fourier m](y)=o(2^n)$. So
    \begin{align*}
        \lim_{n\to\infty}\left|\frac{\Delta^{n}_{h}[\fourier m](y)}{2^n}\right|\leq\lim_{n\to\infty}\int_{\RR}\left|\left(\frac{e^{-2\pi ihx}-1}{2}\right)^ne^{-2\pi iyx}m(x)\right|dx=0.
    \end{align*}
    Therefore
    \begin{align*}
        \lim_{n\to\infty}\frac{\Delta^{n}_{h}[\fourier m](y)}{2^n}=0.
    \end{align*}
    Thus,
    \begin{align*}
        \Delta_{h}^{n}[\fourier m](y)=o(2^n).
    \end{align*}
     Thus,
     \begin{align*}
         \Delta_{h}^{n}[f](y)=\Delta_{h}^{n}[\fourier m](y)+\Delta_{h}^{n}[p](y)=o(2^n)+o(2^n)=o(2^n).
         %\nabla_{h}^{n}[f](y)=\nabla_{h}^{n}[\fourier m](y)+\nabla_{h}^{n}[p](y)=o(2^n)+o(2^n)=o(2^n),\\
         %\delta_{h}^{n}[f](y)=\delta_{h}^{n}[\fourier m](y)+\delta_{h}^{n}[p](y)=o(2^n)+o(2^n)=o(2^n).
     \end{align*}
     As we have proven the claim for forward differences, the claim for backward differences follows by the relationship
     \begin{align*}
        \nabla_{h}^{n}[f](y)=(-1)^n\Delta^{n}_{h}[g](-y),
    \end{align*}
    where $g(y)=f(-y)$. So
    \begin{align*}
        \nabla_{h}^{n}[f](y)=o(2^n),
    \end{align*}
    completing the proof.
\end{proof}

\noindent An application of Theorem 3 yields the following result:

\begin{theorem}
     Let $h\in\RR^+$. Assuming it exists, suppose $\invfourier f$ is a linear combination of Lebesgue measurable functions and elements of the set $$\{\delta^{(m)}(x-a):m\in\NN,a\in\RR\backslash(h^{-1}\ZZ)\}.$$ Let $y\in\RR$. Then
    \begin{align*}
        &\sum_{k=0}^{n}{n\choose k}f(y+kh)=o(2^n),\\
        &\sum_{k=0}^{n}{n\choose k}f(y-kh)=o(2^n),\\
        &\sum_{k=0}^{n}{n\choose k}f\left(y+\left(\frac{n}{2}-k\right)h\right)=o(2^n).
    \end{align*}
\end{theorem}

\begin{proof}
    Let $y\in\RR$. Let $h\in\RR^+$. Suppose that $\invfourier f$ is a linear combination of Lebesgue measurable functions and elements of the set $$\{\delta^{(m)}(x-a):m\in\NN,a\in\RR\backslash(h^{-1}\ZZ)\}.$$ Let $\alpha(y)=e^{i\pi y/h}f(y)$. Note the relationships
    \begin{align*}
        &\Delta^{n}_{h}[\alpha](y)=(-1)^n e^{i\pi y/h}\sum_{k=0}^{n}{n\choose k}f(y+kh),\\
        &\nabla^{n}_{h}[\alpha](y)=e^{i\pi y/h}\sum_{k=0}^{n}{n\choose k}f(y-kh),\\
        &\delta^{n}_{h}[\alpha](y)=e^{i\pi y/h}e^{i\pi n/2}\sum_{k=0}^{n}{n\choose k}f\left(y+\left(\frac{n}{2}-k\right)h\right).
    \end{align*}
    We have that $\invfourier\alpha(x)=\invfourier f(x+h^{-1}/2)$. As translations of Lebesgue measurable functions are Lebesgue measurable, $\invfourier \alpha$ is a linear combination of Lebesgue measurable functions and elements of the set $\{\delta^{(m)}(x-a):m\in\NN,a\in\RR\backslash(h^{-1}/2+h^{-1}\ZZ)\}$. Thus by Theorem 3,
    \begin{align*}
        \Delta_{h}^{n}[\alpha](y)=o(2^n),\\
        \nabla_{h}^{n}[\alpha](y)=o(2^n),\\
        \delta_{h}^{n}[\alpha](y)=o(2^n).
    \end{align*}
    Thus,
    \begin{align*}
        &\sum_{k=0}^{n}{n\choose k}f(y+kh)=o(2^n),\\
        &\sum_{k=0}^{n}{n\choose k}f(y-kh)=o(2^n),\\
        &\sum_{k=0}^{n}{n\choose k}f\left(y+\left(\frac{n}{2}-k\right)h\right)=o(2^n),
    \end{align*}
    as desired.
\end{proof}

Note that the conclusion of the above theorem cannot be improved to $o(\alpha^n)$ for $\alpha<2$. Indeed, take $f(y)=(y+1)^{-1}$. The inverse Fourier transform of $f$ exists and is Lebesgue measurable; however,
\begin{align*}
    \sum_{k\geq0}{n\choose k}\frac{1}{k+1}=\frac{2^{n+1}-1}{n+1},
\end{align*}
which is not $o(\alpha^n)$ if $\alpha<2$. This identity may be verified by observing that
\begin{align*}
    \sum_{k\geq0}{n\choose k}\frac{x^{k+1}}{k+1}=\sum_{k\geq0}{n\choose k}\int_{0}^{x}t^{k}dt=\int_{0}^{x}(t+1)^{n}dt=\frac{(x+1)^{n+1}-1}{n+1},
\end{align*}
where we take $x=1$. 

Unfortunately, we cannot allow all translations of the Dirac delta function and its derivatives in the inverse Fourier transform of $f$ for Theorem 4. This is easy to see: take $f(y)=1$, wherein $\invfourier f(x)=\delta(x)$ and
\begin{align*}
    \sum_{k\geq0}{n\choose k}f(k)=\sum_{k\geq0}{n\choose k}=2^n\neq o(2^n).
\end{align*}
Theorem 4 fails to apply to the above sum since $0\in \ZZ$.

\section{Discussion}

An application of Theorem 1 arises in guaranteeing the rate of convergence of the \emph{Euler transformation} of an alternating series.
\begin{definition}[Euler transformation]
    The Euler transformation of an alternating series $\sum_{n=0}^{\infty}(-1)^n f(n)$ is given by
    \begin{align*}
        \sum_{n=0}^{\infty}(-1)^n f(n)=\sum_{n=0}^{\infty}\frac{(-1)^n\Delta^n[f](0)}{2^{n+1}}.
    \end{align*}
\end{definition}
Indeed, if $\invlaplace f$ is absolutely integrable and is a linear combination of Lebesgue measurable functions and elements of the set $\{\delta^{(m)}(t-a):m\in\NN,a\in\RR^{\geq0}\}$, then the Euler transformation of $\sum_{n=0}^{\infty}(-1)^n f(n)$ converges at least as fast as $\sum_{n=0}^{\infty}(-1)^n/2^{n+1}$, the error of which is $\nicefrac{1}{6}\cdot(\nicefrac{-1}{2})^n$.

Theorem 2 guarantees the existence and convergence of certain Newton series, which can be quite a delicate task otherwise. The proceeding result, Corollary 2, extends work in \cite{entireandexptype} from rational functions with linear denominators to all rational functions with real coefficients.

Theorems 3 and 4 give an asymptotic bound on probabilistic or combinatorial binomial sums when the sequence summed is known (or at least known to be interpolated by a Fourier-transformable function).

\section{Acknowledgements}

Thanks to Sai Sivakumar for directing me to results relevant to the topics of this article and for his excellent mentorship. Thanks to Kenneth DeMason for a thorough review of the paper.

\bibliography{refs}{}
\bibliographystyle{plain}

\end{document}